\documentclass[12pt,oneside]{amsart}
\usepackage{amsmath,amssymb,latexsym,soul,cite,mathrsfs}
\usepackage[pagewise]{lineno}%\linenumbers
\usepackage{color,enumitem,graphicx}
\usepackage[colorlinks=true,urlcolor=blue,
citecolor=red,linkcolor=blue,linktocpage,pdfpagelabels,
bookmarksnumbered,bookmarksopen]{hyperref}
\usepackage[english]{babel}
\usepackage{lineno}

\usepackage[left=2.4cm,right=2.8cm,top=2.5cm,bottom=2.4cm]{geometry}

\usepackage[hyperpageref]{backref}

\numberwithin{equation}{section}
\newtheorem{theorem}{Theorem}[section]
\newtheorem{lemma}[theorem]{Lemma}
\newtheorem{corollary}[theorem]{Corollary}
\newtheorem{proposition}[theorem]{Proposition}

\newtheorem{remark}[theorem]{Remark}

\renewcommand{\epsilon}{\varepsilon}

\renewcommand{\rightarrow}{\to}

\title{A gradient type term for the k-Hessian equation}

\author[M. Cardoso]{Mykael de Ara\'{u}jo Cardoso}

\author[J. \ de Brito Sousa]{Jefferson de Brito Sousa}

\author[J.F.\ de Oliveira]{Jos\'{e} Francisco de Oliveira}

\address[M. Cardoso]{
\newline\indent 
Department of Mathematics
\newline\indent
Federal University of Piau\'{\i}
\newline\indent 
64049-550 Teresina, PI, Brazil}
\email{\href{mailto:mykael@ufpi.edu.br}{mykael@ufpi.edu.br}}
\address[J. de Brito Sousa]{
\newline\indent
Department of Mathematics
\newline\indent
Federal University of Piau\'{\i}
\newline\indent
64049-550 Teresina, PI, Brazil}
\email{\href{mailto:jeffersonbrito2@gmail.com}{jeffersonbrito2@gmail.com}}

\address[J.F.\ de Oliveira]{
\newline\indent Department of Mathematics
	\newline\indent 
	Federal University of Piau\'{i}
	\newline\indent
	64049-550 Teresina, PI, Brazil}
\email{\href{mailto:jfoliveira@ufpi.edu.br}{jfoliveira@ufpi.edu.br}}
\thanks{The first author was supported by Conselho Nacional de Desenvolvimento Cient\'{i}fico e Tecnol\'{o}gico - CNPq and Funda\c{c}\~{a}o de Amparo \`{a} Pesquisa do Estado do Piau\'{i} - FAPEPI}
\thanks{The third author was partially supported by  CNPq grant number 309491/2021-5}

\begin{document}

\maketitle

\begin{abstract}
In this paper, we propose a gradient type term for the $k$-Hessian equation that extends for $k>1$ the classical quadratic  gradient term associated with the Laplace equation. We prove that such as gradient term is invariant by the Kazdan-Kramer change of variables. As applications, we ensure the existence of solutions for a new class of $k$-Hessian equation in the sublinear and superlinear cases for Sobolev type growth. The threshold for existence is obtained in some particular cases.  In addition, for the Trudinger-Moser type growth regime, we also prove the existence of solutions under either subcritical or critical conditions.
\end{abstract}

\section{Introduction}
Let $\Omega\subset \mathbb{R}^n,n\ge 2$ be a bounded smooth domain and $u \in C^2(\Omega)$. For $k=1,2\cdots, n$,  let $S_k[u]$ be the $k$-th Hessian operator  which is defined by
$$
S_k[u]=[D^2 u ]_{k}$$
where $D^2 u$ is  the hessian matrix of $u$ and $[A]_{k}$ denotes the sum of all $k\times k$ principal minors of the matrix $A$. It is worth mentioning that although for $k=1$ we have $S_{1}[u]=\Delta u$ that is a linear operator,  $S_{k}[u]$ with $2\le k\le n$ represents a serial of nonlinear operators linking the Laplacian operator to the Monge Amp\`{e}re operator $S_{n}[u]=\det D^2u$.

To  ensure that the $k$-Hessian operator is elliptic, we must restrict the space function $C^2(\Omega)$. In fact, following \cite{caffarelli} we say a function $u\in C^2(\Omega)\cap C^0(\overline{\Omega})$ is $k$-admissible
function on $\Omega$ if $S_j[u]\ge 0$ for $j=1,\cdots,k$. We denote $\Phi^k(\Omega)$ the space of all $k$-admissible
functions and by $\Phi^{k}_{0}(\Omega)$ the functions $u\in\Phi^{k}(\Omega)$  such that $u_{\mid_{\partial\Omega}}=0$.
Although the $k$-Hessian operators are fully nonlinear, for $2\leq k \leq n$, they have a divergence structure
	$$
		S_k[u]=\dfrac{1}{k}\sum_{i,j=1}^{n}S_{k}^{ij}[u]u_{ij}=\dfrac{1}{k}\sum_{i,j=1}^{n}\partial_{i}(u_{j}S_{k}^{ij}[u]),
		$$ 
  where $u_i=u_{x_i}$, $u_{ij}=u_{x_ix_j}$, and $S_k^{ij}[u]=\dfrac{\partial}{\partial{u_{ij}}}S_k[u]$.
  This leads to various variational and potential theoretic properties which have been studied extensively by many authors and there is a vast literature. We emphasize that the $k$-Hessian operator plays an important role in geometric analysis, partial differential equations, and other branches of modern mathematics, see \cite{B.Guan,wsheng, TR1}.
 Historically, there are a lot of papers in the literature on the existence, regularity, and  qualitative properties of solutions for the $k$-Hessian equation. The pioneer authors L. Caffarelli, L. Nirenberg, J. Spruck \cite{caffarelli}, proved the existence and a priori estimates of smooth solutions to the $k$-Hessian equation. For recent results  we recommend \cite{wang,N.C,N.C2,sobre k-hess5,sobre k-hess6,sobre k-hess2, G.Dai,crit e subcritico} and the references therein.

At this point, we say a few words on the Kazdan-Kramer change of variables and the Laplace equation with quadratic gradient term. To do this, consider the problem
\begin{equation}\label{KK-problem}
    \left\{\begin{aligned}
   &-\Delta u=\vert\nabla u\vert ^2+ f(x,u) & \;\;&\mbox{in}&\;\;\Omega\\
   &u>0& \;\;&\mbox{in}&\;\;\Omega \\
& u=0& \;\;&\mbox{on}&\;\;\partial\Omega 
    \end{aligned}
    \right.
\end{equation}
where $\Delta u=S_{1}[u]$ is the Laplacian and $f:\Omega\times [0,\infty)\to [0,\infty)$ is a continuous function. This type of problem gained prominence after the works \cite{kazdan,MET2,MET14,JEANJEAN 1,JEANJEAN 2, MET1}. In part, this attention is motivated by the observation of Kazdan-Kramer \cite{kazdan} that the problem 
\eqref{KK-problem} can be transformed, through a change of variables, in a problem without dependence on the gradient term $\vert\nabla u\vert^2$ which brakes the line to consider new possibilities. Indeed, the change of variables $v=e^u-1$ turns the problem \eqref{KK-problem} in problem
	\begin{equation}\label{KK-problem2}
	\begin{cases}
		-\Delta v=h(x,v) & \mbox{in}\;\; \Omega\\
		\ v>0 & \mbox{in}\;\;  \Omega \\
		\ v=0 & \mbox{on}\;\; \partial \Omega 
	\end{cases}
	\end{equation}
where $h(x,v)=(1+v)f(x,\log(1+v))$.

In this paper, inspired by \cite{kazdan} and the recent development in \cite{ubilla}, we find a suitable term $H_k=H_k(x,u, \nabla u, D^2 u), k\ge 1$ to play the role of the quadratic term $\vert\nabla u\vert^2$ in \eqref{KK-problem}, when we use the $k$-Hessian operator $S_{k}[u], k\ge 2$ instead of the Laplacian $\Delta u$. In fact, we consider the equation
\begin{equation}
	\begin{cases}\label{equação 1.1}
		S_k[u]=g(u)H_{k}(x,u, \nabla u, D^2 u)+ f(x,u) & \mbox{in} \;\; \Omega\\
		\ u<0 & \mbox{in}\;\;  \Omega \\
		\ u=0  & \mbox{on}\;\;\partial \Omega 	
	\end{cases}
\end{equation}
where $g:(-\infty, 0]\to [0,\infty)$ and $f:\overline{\Omega}\times\mathbb{R}\to\mathbb{R}$ are continuous functions and the ``$k$-gradient natural term'' $H_k$ is given by
 \begin{equation}\label{Hk}
    H_k(x,u,\nabla u, D^{2}u)=\sum_{i=1}^{\binom{n}{k}}\sum_{t=1}^{k} \det([D^{2}u]_{i}\xleftarrow{t}u_{x_{i_t}}\nabla^{i}_{k}u),\;\;\mbox{with}\;\; \nabla^{i}_{k}u=\left[\begin{array}{c}
         u_{x_{i_1}}  \\
        u_{x_{i_2}}\\
         \vdots\\
         u_{x_{i_k}}
    \end{array}  \right],
 \end{equation} 
in which $[D^2u]_{i}$ , with  $i=1,...,$ $\binom{n}{k}$ represents the $i$-th principal submatrix of $D^2u$ of size $k$ determined by the index set $\alpha_i=\{i_1,\cdots, i_k\}\subset\{1, \cdots, n\}$   and $([D^{2}u]_{i}\xleftarrow{t}u_{x_{i_t}}\nabla^{i}_{k}u)$ denotes the matrix whose $t$-th column is $u_{x_{i_t}}\nabla^{i}_{k}u$ and whose remaining columns coincide with those of $[D^2u]_{i}$, see Lemma~\ref{lema1} below for details.

We note that, if $k=1$, the hessian matrix  $D^{2}u$ has exactly  $n$ principal submatrix of size $1$, $u_{x_{i_1}}\nabla^{i}_{1}u=[u^2_{x_{i_1}}]$  and  $([D^{2}u]_{i}\xleftarrow{1}u_{x_{i_1}}\nabla^{i}_{1}u)=[u^2_{x_{i_1}}]$. Hence,  $H_1(x,u,\nabla u, D^{2}u)=\vert\nabla u\vert^{2}$. In this sense,  the $k$-Hessian equation \eqref{equação 1.1} with the gradient type term $H_k$ in \eqref{Hk} extends  the problem \eqref{KK-problem} for $k\ge 2$.
Furthermore,  we prove that (cf. Section~\ref{section3}) the change of variable of Kazdan-Kramer type 
\begin{equation}\label{Ag-sec1}
    A_g(s)=-\int_{s}^{0}e^{G(t)}dt, \;\;\mbox{with}\;\; G(t)=\int_{t}^{0}g(\tau)d\tau
\end{equation} 
transforms the  equation \eqref{equação 1.1} into the $k$-Hessian equation without the gradient type term $H_k$
\begin{equation}
	\begin{cases}\label{equação 1.2}
		S_k[v]=h(x,v) & \mbox{in}\;\; \Omega\\
		\ v<0 & \mbox{in}\;\;  \Omega \\
		\ v=0 & \mbox{on}\;\; \partial \Omega,
	\end{cases}
\end{equation}
where 
\begin{equation}\label{h-transformed}
    h(x,s)=e^{kG(A^{-1}_{g}(s))}f(x,A^{-1}_{g}(s)).
\end{equation}
We will establish existence of solutions for the new $k$-Hessian equation \eqref{equação 1.1}  for  both Sobolev $k<n/2$ (cf. \cite{tso}) and Trudinger-Moser $k=n/2$ (cf.\cite{Tian}) cases. Non-existence result for $k<n/2$ is also proven.

\section{Statements of results}
Throughout this paper, the functions $f$ and $g$ are assumed to satisfy the following conditions:
\begin{itemize}
    \item [$(H_g)$] $g:(-\infty, 0]\to [0, \infty)$ is a continuous function
    \item [$(H_f)$] $f\in C^{0}(\overline{\Omega \times \mathbb{R}^{-}}) $, with $f(x,z)>0$ for $z<0$.
\end{itemize}
Let $\lambda_1>0$ be the first eigenvalue of the $k$-Hessian operator (cf. \cite{citekey 3}). We say that the pair $f,g$  is sublinear  if
\begin{equation}\label{f-sublinear}
    \lim_{z\rightarrow -\infty}\dfrac{e^{kG(z)}f(x,z)}{\left(\int_{z}^{0}e^{G(s)}ds\right)^{k}}<\lambda_1,\;\; \mbox{uniformly on $\overline{\Omega}$}
\end{equation}
in case it occurs
\begin{equation}\label{f-superlinear}
    \lim_{z\rightarrow -\infty}\dfrac{e^{kG(z)}f(x,z)}{\left(\int_{z}^{0}e^{G(s)}ds\right)^{k}}>\lambda_1,\;\; \mbox{uniformly on $\overline{\Omega}$}
\end{equation}
we say that $f, g$ is superlinear.

The conditions \eqref{f-sublinear} and \eqref{f-superlinear} ensure that $h$ given by \eqref{h-transformed} is sublinear and superlinear respectively, in the sense established in \cite{sobre k-hess3}, see Lemma~\ref{sub-super} below. Further, if $f$ is superlinear then the pair $f,g$ is superlinear too, see Remark~\ref{remark2}.
\subsection{The Sobolev case}
Throughout this section, we assume $\Omega\subset\mathbb{R}^n$ (bounded)  of class $C^{3,1}$ and the Sobolev condition $k<n/2$. In addition, we suppose  the following:
\begin{itemize}
    \item [$(H_{SC})$] If $k^*=n(k+1)/(n-2k)$ (cf. \cite{citekey 3,tso}) is the critical exponent for the embedding $\Phi^{k}_{0}(\Omega)\hookrightarrow L^{p+1}(\Omega)$, then we have
    \begin{equation}
        \lim_{z\rightarrow -\infty}\dfrac{e^{kG(z)}f(x,z)}{\left(\int_{z}^{0}e^{G(s)}ds\right)^{k^*-1}}=0,\;\; \mbox{uniformly on $\overline{\Omega}$},
    \end{equation}
	where $G(z)=\int_{z}^{0}g(\tau)d\tau$.
	\item [$(H_{AR})$] There exist constants $\theta>0$ and large $M$ such that
	\begin{equation}
		\dfrac{k+1}{1-\theta}\int_{z}^{0}e^{(k+1)G(s)}f(x,s) ds\leq e^{kG(z)}f(x,z)\int_{z}^{0}e^{G(s)}ds
	\end{equation}
	for any $z<-M$.
\end{itemize}
First, we  consider the superlinear case. 
\begin{theorem}[superlinear case]\label{thm2-sup} Suppose $f$ and $g$ satisfying  $(H_f)$, $(H_g)$ and such that the function $h$ given by \eqref{h-transformed} belongs to $C^{1,1}(\overline{\Omega \times \mathbb{R}^{-}})$. In addition, assume \eqref{f-superlinear}, $(H_{SC})$, $(H_{AR})$ and the following
\begin{equation}\label{s0}
\lim_{ z\rightarrow 0^-}\dfrac{f(x,z)}{\vert z\vert ^k}<\lambda_{1}\;\; \mbox{uniformly on $\overline{\Omega}$}.
\end{equation}
Then, the problem \eqref{equação 1.1} has a nontrivial solution $u\in C^{2}(\Omega)\cap C^{0}(\overline{\Omega})$.
\end{theorem}
For the sublinear case we prove the following:
\begin{theorem}[sublinear case]\label{thm1-sub}	Suppose $f$ and $g$ satisfying  $(H_f)$, $(H_g)$ and such that the function $h$ given by \eqref{h-transformed} belongs to $C^{1,1}(\overline{\Omega} \times \mathbb{R}^{-})\cap C^{0}(\overline{\Omega\times\mathbb{ R^{-}}})$. Furthermore, assume \eqref{f-sublinear} and the following
\begin{equation}\label{s011}
\lim_{ z\rightarrow 0^-}\dfrac{f(x,z)}{\vert z\vert ^k}>\lambda_{1}\;\; \mbox{uniformly on $\overline{\Omega}$}.
\end{equation}
Then,  the problem \eqref{equação 1.1} has a nontrivial solution $u\in C^{2}(\Omega)\cap C^{0}(\overline{\Omega})$.
\end{theorem}
Theorems \ref{thm2-sup} and \ref{thm1-sub} improve and complement \cite[Theorems~ 1.2 and 1.3]{sobre k-hess3} by including the gradient type term $H_k$. Also, these results complement for $k>1$ earlier  ones for the Laplace operator $k=1$ in \cite{MET2,MET14,JEANJEAN 1, JEANJEAN 2}. For related problems involving  
$p$-Laplacian operator we recommend \cite{ubilla,MET21,MET30} and references therein.

We observe that the pair of functions $f(x,z)=(e^{-z}-1)^pe^{kz}$  and $g\equiv 1$ satisfies the assumptions of Theorem~\ref{thm1-sub} for $1\le p<k$ and  the hypothesis of Theorem~\ref{thm2-sup} if $p \in (k,k^{*}-1)$. Thus, as by-product of Theorems \ref{thm2-sup} and \ref{thm1-sub} we have the following:
\begin{corollary}\label{co-exist} Assume  $f(x,z)=(e^{-z}-1)^pe^{kz}$ for $1\le p <k^*-1$, $p\not=k$  and $g\equiv 1$. Then the problem \eqref{equação 1.1} admits  a nontrivial solution $u\in C^{2}(\Omega)\cap C^{0}(\overline{\Omega})$    
\end{corollary}
The non-existence result below shows that the strict condition $p<k^{*}-1$ in the Corollary~\ref{co-exist} is the threshold for the existence at least  for that pair of $f$ and $g$. 

\begin{theorem}\label{non-ex}
 	Suppose that  $\Omega$ is star-shaped with respect to the origin and let $f(x,z)=(e^{-z}-1)^pe^{kz}$ and $g\equiv1$. Then the problem \eqref{equação 1.1} has no negative solution in $C^1(\overline{\Omega})\cap C^4(\Omega)$ when $p\geq k^*-1$.
\end{theorem}
The proof of  Theorem~\ref{non-ex} relies on the Pucci-Serrin general identity \cite{Pucci-Serrin} which can be applied for a general pair $f,g$. We choose the pair  $f(x,z)=(e^{-z}-1)^pe^{kz}$ and $g\equiv1$ just to be clear that, in some sense, the gradient type term in \eqref{Hk} preserves the the criticality derived from the exponent $k^*$, see for instance \cite{tso}. 
\begin{remark}
  In Theorem~\ref{non-ex} we are assuming the boundary condition $\langle x,\nu\rangle>0$ on $\partial \Omega$, where  $\nu$ is the unit outer normal, but it is necessary only for the limit case $p=k^*-1$. See Lemma~\ref{lema-nonex} below. 
\end{remark}
\subsection{The Trudinger-Moser case}
\noindent In this section, we assume $\Omega=B\subset\mathbb{R}^n$ the unit ball and the Trudinger-Moser  condition $k=n/2$.  
\begin{enumerate}
    \item [$(\mathcal{H}_f)$] $f\in C^{0}(\overline{B} \times \mathbb{R}, [0,\infty)) $, $x\mapsto f(x, s)$ is a radially symmetric function, and $f\equiv 0$ on $\overline{B} \times [0, \infty)$.
     \item [$(\mathcal{H}_{AR})$] There exist  $\vartheta>k+1$, $0<r_1<r_2<1$ and $z_0>0$ such that for $z<-z_0$
        \begin{equation}\nonumber
     \begin{aligned}
		&\vartheta \int_{z}^{0}e^{(k+1)G(s)}f(x,s) ds\leq e^{kG(z)}f(x,z)\int_{z}^{0}e^{G(s)}ds &\mbox{if} &\;x\in\overline{B}\\
		&\int_{z}^{0}e^{(k+1)G(s)}f(x,s) ds>0\;\;&\mbox{if} &\; x\in B_{r_2}\setminus B_{r_1},
		\end{aligned}
	\end{equation}
		where $G(z)=\int_{z}^{0}g(\tau)d\tau$.
    \item [$(\mathcal{H}_{AR_1})$] There exist constants $L, M > 0$ such that  
    $$0 < \int_{z}^{0}e^{(k+1)G(s)}f(x,s) ds \leq M e^{kG(z)}f(x,z),$$
    for all $z< -L$ and $x \in \overline{B}$.
\end{enumerate}
\noindent Inspired by \cite{DMRuf} we say that the pair $f,g$ has subcritical exponential growth if for any $\alpha>0$
\begin{equation}\label{e-sub}
    \lim_{ s\rightarrow -\infty} \dfrac{e^{kG(s)}f(x,s)}{e^{\alpha\left(\int_s^0 e^{G(t)}dt\right)^{\frac{n+2}{n}}}}=0,\;\; \mbox{uniformly on $\overline{B}$.}
\end{equation}
Further, the pair $f,g$ has critical exponential growth  if there exists $\alpha_0>0$ such that
\begin{equation}\label{fe-cri}
\lim_{s\rightarrow -\infty} \dfrac{e^{kG(s)}f(x,s)}{e^{\alpha \left(\int_s^0 e^{G(t)}dt\right)^{\frac{n+2}{n}}}}=\left\{
\begin{aligned}
           &0,   \;\;&\mbox{for all}&\;\;  \alpha>\alpha_0\\
		&+\infty, \;\;&\mbox{for all}&\;\;  \alpha<\alpha_0  \\
\end{aligned}
\right.
\end{equation}
uniformly for $x \in \overline{B}$.

This notion of criticality is motivated by the Trudinger-Moser type inequality for $k$-Hessian operator, see \cite{Tian}. Actually, we have that the condition \eqref{e-sub} (or  \eqref{fe-cri}) on the pair $f,g$  means that the transformed function $h$ in \eqref{h-transformed} has subcritical exponential growth (or critical exponential growth) in the sense established in \cite{DMRuf,crit e subcritico}, see Lemma~\ref{H-FG}-$(a)$ below.

Denote by $X_0$ the set of all locally absolutely continuous functions $u:(0,1]\rightarrow \mathbb{R}$ satisfying $u(1)=0$ and $ \int_{0}^{1} r^{n-k}\vert u^{\prime}\vert^{k+1}dr<\infty$. Let us consider the constant $\Lambda_1>0$ (cf. \cite{crit e subcritico}) defined by 
\begin{equation}\label{Lambda1}
    \Lambda_1= \inf_{u \in X_0\setminus\{0\} }   \frac{c_n\int_0^1 r^{n-k}\vert u^{\prime}\vert^{k+1}dr}{\tau\int_0^1 r^{n-1}\vert u\vert ^{k+1}dr},
\end{equation}
where $c_ n=\frac{\omega_{n-1}}{k}\binom{n-1}{k-1}$, with $\omega_{n-1}$ being the surface area of the unit sphere in $\mathbb{R}^n$ and $\tau=\omega_{n-1}$.
\begin{theorem}[subcritical case] \label{thm2} Assume that $f$ and $g$ satisfies $(H_g)$ $(\mathcal{H}_f)$, $(\mathcal{H}_{AR})$ and \eqref{e-sub}. Furthermore, assume that
\begin{equation}\label{super00}
 \limsup_{s\rightarrow 0^{-}} \dfrac{e^{kG(s)}f(x,s)}{\left(\int_s^0 e^{G(t)}dt\right)^{k}}<\Lambda_1.
 \end{equation}
  Then (\ref{equação 1.1}) admits a radially symmetric  solution $u\in C^{2}(B)\cap C^{0}(\overline{B})$.
\end{theorem}
\noindent In order to ensure our existence result for critical exponential growth, we also need the condition
\begin{equation}\label{fmin-max}
    \lim_{s \rightarrow - \infty} \dfrac{e^{kG(s)}\vert s\vert f(x,s)}{e^{\alpha_0 \left(\int_s^0 e^{G(t)}dt\right)^{\frac{n+2}{n}}}}=b_0>\dfrac{1}{e^{1+\frac{1}{2}+\cdots+\frac{1}{k}}}\left(\dfrac{\alpha_n}{\alpha_0}\right)^{\frac{n}{2}}\dfrac{n}{\tau}
\end{equation}
uniformly in $x \in \overline{B}$, where $\alpha_n=n[\frac{\omega_{n-1}}{k}\binom{n-1}{k-1}]^{\frac{2}{n}}=nc_n^{\frac{2}{n}}$ is the critical Moser's constant for Trudinger-Moser type inequality for $k$-Hessian operator, see \cite{Tian}.
\begin{theorem}[critical case]\label{thm3}
  Assume $(H_g)$, $(\mathcal{H}_f)$, $(\mathcal{H}_{AR})$, $(\mathcal{H}_{AR_1})$,  \eqref{fe-cri}, \eqref{super00} and \eqref{fmin-max}.
  Then (\ref{equação 1.1}) admits a radially symmetric solution $u\in C^{2}(B)\cap C^{0}(\overline{B})$.
\end{theorem}
\section{The change of variable of Kazdan-Kramer type for k-Hessian equation}
\label{section3}
First of all, we recall some notations introduced in \cite{Horn}. Denote by $M_n(\mathbb{R})$ the set of all $n\times n$ matrices over $\mathbb{R}$. Let $A\in M_n(\mathbb{R})$. For index sets $\alpha,\beta\subset\{1,\cdots, n\}$ we denote by 
$A[\alpha,\beta]$ the submatrix of entries that lie in the rows of 
$A$ indexed by $\alpha$ and the columns indexed by $\beta$. If $\alpha=\beta$ the submatrix $A[\alpha, \alpha]$ is a principal submatrix of $A$. An $n\times n$ 
matrix has $\binom{n}{k}$ distinct principal submatrices of size $k$, that is,  for which the cardinality $\vert \alpha\vert=k$.  By convenience, if  $A=[a_1\cdots a_m]\in M_{m}(\mathbb{R})$ is partitioned according to its columns and let $b\in\mathbb{R}^m$, we define
\begin{equation}\nonumber
    (A\xleftarrow{i} b)=[a_1\cdots a_{i-1}\, b \,a_{i+1}\cdots a_m]
\end{equation}
that is, $(A\xleftarrow{i} b)$ denotes the matrix whose $i$-th column is $b$ and whose remaining columns coincide with those of $A$. With the above notation, for $\Omega\subset \mathbb{R}^n,n\ge 2$ bounded smooth domain,  $u \in C^2(\Omega)$ and $1\le k
\le n$, the $k$-th Hessian operator is given by 
\begin{equation}\nonumber
  S_k[u]=\sum_{\vert \alpha\vert=k}\det D^2u[\alpha,\alpha]
\end{equation}
in which the sum is over all $\binom{n}{k}$ index sets $\alpha\subset\{1,\cdots, n\}$ with $\vert\alpha\vert=k$. 

The next result was essentially proved in \cite[Lemma~2.3]{sobre k-hess4} and tells us precisely how the composition functions affects the $k$-Hessian operator and  provide the key to find the gradient type term for the $k$-Hessian equation in \eqref{Hk}. By completeness and to set the notation we include the proof.
	\begin{lemma}\label{lema1}
	Let $u \in C^2(\Omega)$  and let $A:I\to\mathbb{R}$ be a $C^2$-function defined on an interval $I\supset\left\{u(x)\,:\, x\in \Omega\right\}$. Then
	$$S_k[A(u)]=S_k[u][A^\prime(u)]^k + [A^\prime(u)]^{k-1}A^{\prime\prime}(u) H_k(x,u,\nabla u, D^{2}u),\;\; k=1,2, \dots,n$$
with
 \begin{equation}\nonumber
    H_k(x,u,\nabla u, D^{2}u)= \sum_{i=1}^{\binom{n}{k}}\sum_{t=1}^{k} \det(D^{2}u[\alpha_i,\alpha_i]\xleftarrow{t}u_{x_{i_t}}\nabla^{i}_{k}u),\;\;\mbox{with}\;\; \nabla^{i}_{k}u=\left[\begin{array}{c}
         u_{x_{i_1}}  \\
        u_{x_{i_2}}\\
         \vdots\\
         u_{x_{i_k}}
    \end{array}  \right],
 \end{equation} 
 where $\alpha_i=\{i_1,\cdots, i_k\}\subset\{1, \cdots, n\}$ is index set with $\vert \alpha_i\vert=k$.
\end{lemma}
\begin{proof}
For $i=1,2, \cdots$, $\binom{n}{k}$, let $\alpha_i=\{i_1,\cdots, i_k\}\subset\{1, \cdots, n\}$ index set with $\vert\alpha_i\vert=k$. Hence, we have corresponding submatrix $D^2(A(u))[\alpha_i,\alpha_i]=[d_{i_1}\,d_{i_2}\cdots d_{i_k}]$ with each column $d_{i_t}, 1\le t\le k$  of the form $$d_{i_t}=A^{\prime\prime}(u)d^{1}_{i_t}+A^{\prime}(u)d^2_{i_t},$$ where
\begin{align*}
    d^{1}_{i_t}=\left[\begin{array}{c}
         u_{x_{i_1}}u_{x_{i_t}}  \\
        u_{x_{i_2}}u_{x_{i_t}}\\
         \vdots\\
         u_{x_{i_k}}u_{x_{i_t}}
    \end{array}  \right] \;\;\mbox{and}\;\; d^{2}_{i_t}=\left[\begin{array}{c}
         u_{x_{i_1}x_{i_t}}  \\
        u_{x_{i_2}x_{i_t}}\\
         \vdots\\
         u_{x_{i_k}x_{i_t}}
    \end{array}  \right].
\end{align*}
Since that the function $B\to \det B$ is multilinear in the columns of $B$, we have 
\begin{align*}
    \det D^2(A(u))[\alpha_i,\alpha_i]&=\det [d_{i_1}\;d_{i_2}\;\cdots\; d_{i_k}]\\
   & =A^{\prime\prime}(u)\det [d^1_{i_1}\;d_{i_2}\cdots\; d_{i_k}]+A^{\prime}(u)\det [d^2_{i_1}\;d_{i_2}\cdots\; d_{i_k}]\\
   &=(A^{\prime\prime}(u))^2\det [d^1_{i_1}\;d^1_{i_2}\cdots\; d_k]+A^{\prime\prime}(u)A^{\prime}(u)\det [d^1_{i_1}\;d^2_{i_2}\cdots\; d_{i_k}]\\
   & +A^{\prime}(u)A^{\prime\prime}(u)\det [d^2_{i_1}\;d^{1}_{i_2}\cdots\; d_{i_k}]+(A^{\prime}(u))^2\det [d^2_{i_1}\;d^2_{i_2}\cdots\; d_{i_k}].
\end{align*}
Noticing that the columns $d^{1}_{i_1}$ and $d^{1}_{i_2}$ are proportional we can write 
\begin{align*}
    \det D^2(A(u))[\alpha_i,\alpha_i]&=A^{\prime\prime}(u)A^{\prime}(u)\det [d^1_{i_1}\;d^2_{i_2}\cdots\; d_{i_k}]
    +A^{\prime}(u)A^{\prime\prime}(u)\det [d^2_{i_1}\;d^{1}_{i_2}\cdots\; d_{i_k}]\\
    &+(A^{\prime}(u))^2\det [d^2_{i_1}\;d^2_{i_2}\cdots\; d_{i_k}].
\end{align*}
For $s\not=t$ the columns $d^{1}_{i_s}$ and $d^{1}_{i_t}$ are also proportional, then we can repeat the above argument to conclude that
\begin{equation}\nonumber
\begin{aligned}
     \det D^2(A(u))[\alpha_i,\alpha_i]&=(A^{\prime}(u))^k\det D^2u[\alpha_i,\alpha_i]\\
      & +A^{\prime\prime}(u)(A^{\prime}(u))^{k-1}\sum_{t=1}^{k} \det(D^{2}u[\alpha_i,\alpha_i]\xleftarrow{t}d^{1}_{i_t}).
      \end{aligned}
\end{equation}    
Hence, 
\begin{align*}
    S_k[A(u)]&=\sum_{i=1}^{\binom{n}{k}} \det D^2(A(u))[\alpha_i,\alpha_i]=(A^{\prime}(u))^k \sum_{i=1}^{\binom{n}{k}}\det D^2u[\alpha_i,\alpha_i]\\
    &+A^{\prime\prime}(u)(A^{\prime}(u))^{k-1}\sum_{i=1}^{\binom{n}{k}}\sum_{t=1}^{k} \det(D^{2}u[\alpha_i,\alpha_i]\xleftarrow{t}d^{1}_{i_t})\\
    &=(A^{\prime}(u))^kS_k[u]+A^{\prime\prime}(u)(A^{\prime}(u))^{k-1}\sum_{i=1}^{\binom{n}{k}}\sum_{t=1}^{k} \det(D^{2}u[\alpha_i,\alpha_i]\xleftarrow{t}d^{1}_{i_t}).
\end{align*}
\end{proof}
\noindent Let $g:(-\infty, 0] \to [0, \infty)$ be continuous function. The change of variables of Kazdan-Kramer  type associated to $g$ is the $C^2$-diffeomorphish $A_g:(-\infty, 0]\to (-\infty, 0]$ given by
\begin{equation}\label{AG-change}
    A_g(s)=-\int_{s}^{0}e^{G(t)}dt, \quad\mbox{with}\;\; G(t)=\int_{t}^{0}g(\tau)d\tau. 
\end{equation} 
In view of the result below, we can see that \eqref{AG-change} transforms the $k$-Hessian equation with gradient type term \eqref{equação 1.1} into the $k$-Hessian equation \eqref{equação 1.2} without that term.
\begin{proposition}\label{proposição 2.1}
Let  $H_k$ and $A_g$  be given by \eqref{Hk} and \eqref{AG-change}, respectively. Suppose that $f\in C^{0}(\overline{\Omega\times\mathbb{R}^{-}}, [0,\infty))$. Then,  for each  solution $u\in C^{2}(\Omega)\cap C^{0}(\overline{\Omega})$ of \eqref{equação 1.1}, $v=A_g(u)$ is a solution of class $C^{2}(\Omega)\cap C^{0}(\overline{\Omega})$  of the equation \eqref{equação 1.2}  with 	
\begin{equation}\label{h-transformed2x}
    h(x,s)=e^{kG(A^{-1}_{g}(s))}f(x,A^{-1}_{g}(s)).
\end{equation}
Reciprocally, if $v\in C^{2}(\Omega)\cap C^{0}(\overline{\Omega})$ is a solution  of \eqref{equação 1.2} with $h(x,s)$ of the form \eqref{h-transformed2x}  then $u=A^{-1}_{g}(v)$ is a solution of \eqref{equação 1.1} with $u\in C^{2}(\Omega)\cap C^{0}(\overline{\Omega})$.
\end{proposition} 
\begin{proof}
It is easy to verify that $A_g$ is a $C^2$-diffeomorphish such that $A_g(0)=0$, $A_g(-\infty)=-\infty$,  and $A_g<0$, $A^{\prime}_{g}>0$ on $(-\infty, 0)$. In addition, its solves the ordinary differential equation 
\begin{equation}\label{ODE}
(y^{\prime}(s))^{k-1}y^{\prime\prime}(s)+g(s)(y^{\prime}(s))^{k}=0.
\end{equation} 
If $u\in C^{2}(\Omega)\cap C^0(\overline{\Omega})$ solves \eqref{equação 1.1}, we have that $v=A_g(u)$ is such that $v\in C^{2}(\Omega)\cap C^0(\overline{\Omega})$,  $v<0$ in $\Omega$ and $v_{\vert_{\partial\Omega}}=0$. Moreover, since $A_g$ solves \eqref{ODE}, Lemma~\ref{lema1} yields
\begin{equation}\nonumber
    \begin{aligned}
    S_k[A_g(u)]&=\left[[A^{\prime}_{g}(u)]^{k-1}A^{\prime\prime}_{g}(u)+g(u)[A^{\prime}_{g}(u)]^k\right]H_{k}(x,u, \nabla u, D^2 u)+[A^{\prime}_{g}(u)]^kf(x,u)\\
    &=[A^{\prime}_{g}(u)]^kf(x,u).\\
    \end{aligned}
\end{equation}
Noticing that $A^{\prime}_{g}(s)=e^{G(s)}$ and $u=A^{-1}_g(v)$ we obtain that $v$ solves $$S_k[v]=e^{kG(A^{-1}_{g}(v))}f(x,A^{-1}_{g}(v)).$$ Reciprocally, if $v\in C^{2}(\Omega)\cap C^0(\overline{\Omega})$ solves equation \eqref{equação 1.2} with $h$ given by \eqref{h-transformed}, then $u=A^{-1}_g(v)$ is such that $u\in C^{2}(\Omega)\cap C^0(\overline{\Omega})$,  $u<0$ in $\Omega$ and $u_{\vert_{\partial\Omega}}=0$. In addition, we have
\begin{equation}\nonumber
\begin{aligned}
    S_k[A_g(u)]=S_k[v]=h(x, v)& =e^{kG(A^{-1}_g(v))}f(x,A^{-1}_{g}(v))
    =e^{kG(u)}f(x, u)
\end{aligned}
\end{equation}
and, from Lemma~\ref{lema1} 
\begin{equation}\nonumber
\begin{aligned}
    S_k[A_g(u)]=S_k[u][A^\prime_{g}(u)]^k + [A^\prime_{g}(u)]^{k-1}A^{\prime\prime}_{g}(u) H_k(x,u,\nabla u, D^{2}u).
\end{aligned}
\end{equation}
Consequently, by using  $A^{\prime\prime}_{g}(s)=-g(s)e^{G(s)}$ and comparing the last two identities, we have 
\begin{equation}\nonumber
\begin{aligned}
    S_k[u]& =[A^\prime_{g}(u)]^{-k} e^{kG(u)}f(x, u) - [A^\prime_{g}(u)]^{k-1}[A^\prime_{g}(u)]^{-k}A^{\prime\prime}_{g}(u) H_k(x,u,\nabla u, D^{2}u)\\
    &=g(u)H_k(x,u,\nabla u, D^{2}u)+f(x, u).
\end{aligned}
\end{equation}
Hence, $u=A^{-1}_{g}(v)$ solve the equation \eqref{equação 1.1}.
\end{proof}
\section{Existence for the Sobolev case: Proof of Theorem~\ref{thm2-sup} and Theorem~\ref{thm1-sub}}
First, we prove Theorem~\ref{thm2-sup}. The proof relies on the Proposition~\ref{proposição 2.1} combined with the existence result in \cite[Theorem~1.2]{sobre k-hess3}. In fact,  according to \cite{sobre k-hess3} the $k$-Hessian equation
\begin{equation}\label{HessianPsi}
	\begin{cases}
		S_k[u]=\psi(x,u) & \mbox{in}\;\; \Omega\\
		\ u<0 & \mbox{in}\;\;  \Omega \\
		\ u=0 & \mbox{on}\;\; \partial \Omega,
	\end{cases}
\end{equation}
admits $k$-admissible solution  in $C^{3,\alpha}(\Omega)\cap C^{0,1}(\overline{\Omega})$ for some $\alpha\in (0,1)$ if $\Omega$ is of class $C^{3,1}$ and $\psi:\overline{\Omega}\times (-\infty, 0]\to [0, \infty)$ satisfies the assumptions $\psi\in C^{1,1}(\overline{\Omega\times\mathbb{ R^{-}}})$, with $\psi(x,z)>0$ for $z<0$, and the following
\begin{equation}\label{psis0}
\lim_{ z\rightarrow 0^-}\dfrac{\psi(x,z)}{\vert z\vert ^k}<\lambda_{1}\;\; \mbox{uniformly on $\overline{\Omega}$},
\end{equation}
\begin{equation}\label{psi-superlinear}
\lim_{ z\rightarrow -\infty}\dfrac{\psi(x,z)}{\vert z\vert ^k}>\lambda_{1}\;\; \mbox{uniformly on $\overline{\Omega}$},
\end{equation}
\begin{equation}\label{psi-subcritical}
	\lim\limits_{z\rightarrow -\infty}\dfrac{\psi(x,z)}{\vert z\vert ^{k^*-1}}=0, 
	\end{equation}
	and, for some constants $\theta>0$ and large $M$ 
	\begin{equation}\label{psi-AR}
		\int_{z}^{0} \psi(x,s) ds\leq \dfrac{1-\theta}{k+1}\vert z\vert \psi(x,z),
	\end{equation}
for any $z<-M$.

Thus, in view of Proposition~\ref{proposição 2.1}, in order to ensure  Theorem~\ref{thm2-sup}  and since we are assuming  that the transformed function $h$ given by \eqref{h-transformed} is such that $h\in C^{1,1}(\overline{\Omega\times\mathbb{ R^{-}}})$ with $h(x,z)>0$ for $z<0$, 
we only need to verify  that $h$ satisfies the conditions \eqref{psis0}, \eqref{psi-superlinear}, \eqref{psi-subcritical} and \eqref{psi-AR}. Next, the Lemma~\ref{sub-super} ensures that $h$ satisfies  \eqref{psis0} and \eqref{psi-superlinear} and the Lemma~\ref{HAR} shows that $h$ satisfies \eqref{psi-subcritical} and \eqref{psi-AR}.

\begin{lemma}\label{sub-super} Let $h$ be given by \eqref{h-transformed}. Then, we have the identities
\begin{equation}\nonumber
\lim_{ z\rightarrow 0^-}\dfrac{h(x,z)}{\vert z\vert ^k}=\lim_{z\rightarrow 0^-}\dfrac{f(x,z)}{\vert z\vert ^k}
\end{equation}
and 
\begin{equation}\nonumber
  \lim_{ z\rightarrow -\infty}\dfrac{h(x,z)}{\vert z\vert ^k}=\lim_{t\rightarrow -\infty}\dfrac{e^{kG(t)}f(x,t)}{\left(\int_{t}^{0}e^{G(s)}ds\right)^{k}}.
\end{equation}
In particular,  if the pair $f, g$ is such that \eqref{s0} and \eqref{f-superlinear} hold, then $h$  satisfies \eqref{psis0} and \eqref{psi-superlinear}.
\end{lemma}
\begin{proof} By the L'Hospital rule, we have
\begin{equation}\label{H-id}
\begin{aligned}
   \lim_{t\rightarrow 0^-}\dfrac{tA^{\prime}_{g}(t)}{A_{g}(t)} & =\lim_{ t\rightarrow 0^-}\dfrac{te^{G(t)}}{-\int_{t}^{0}e^{G(s)} ds}\\
   &= \lim_{ t\rightarrow 0^-}\dfrac{e^{G(t)}-tg(t)e^{G(t)}}{e^{G(t)}}\\
   &= \lim_{ t\rightarrow 0^-}[1-tg(t)]\\
   &=1.
\end{aligned}
\end{equation}
By setting $t=A^{-1}_{g}(z)$ we have $t\rightarrow 0^-$ if $z\to 0^{-}$. Thus, from \eqref{H-id}
\begin{equation}\nonumber
    \begin{aligned}
       \lim_{ z\rightarrow 0^-}\dfrac{h(x,z)}{\vert z\vert^k}&=\lim_{z\rightarrow 0^-}\frac{e^{kG(A^{-1}_{g}(z))}f(x,A^{-1}_{g}(z))}{\vert z\vert^{k}}\\
       &= \lim_{t\rightarrow 0^-}\frac{e^{kG(t)}f(x,t)}{\vert A_g(t)\vert^{k}}\\
       &= \lim_{t\rightarrow 0^-}\frac{[A^{\prime}_{g}(t)]^{k}f(x,t)}{\vert A_g(t)\vert^{k}}\\
       &= \lim_{t\rightarrow 0^-} \Big[\frac{-t A^{\prime}_g(t)}{-A_g(t)}\Big]^{k}\frac{f(x,t)}{\vert t\vert^{k}}\\
       &= \lim_{t\rightarrow 0^-}\frac{f(x,t)}{\vert t\vert ^{k}}.
    \end{aligned}
\end{equation}
 Analogously, by setting $t=A^{-1}_{g}(z)$ we have $t\rightarrow -\infty $ if $z\to -\infty$. Thus, 
	\begin{equation}\nonumber
    \begin{aligned}
       \lim_{ z\rightarrow -\infty}\dfrac{h(x,z)}{\vert z\vert ^k}&=\lim_{z\rightarrow -\infty}\frac{e^{kG(A^{-1}_{g}(z))}f(x,A^{-1}_{g}(z))}{\vert z\vert ^{k}}\\
       &=\lim_{t\rightarrow -\infty}\dfrac{e^{kG(t)}f(x,t)}{\left(\int_{t}^{0}e^{G(s)}ds\right)^{k}}.
    \end{aligned}
\end{equation}
\end{proof}
\begin{remark} \label{remark2} Note that  
\begin{equation}\nonumber
 \lim_{z\rightarrow -\infty}\dfrac{e^{kG(z)}f(x,z)}{\left(\int_{z}^{0}e^{G(s)}ds\right)^{k}}\ge  \lim_{z\rightarrow -\infty}\dfrac{f(x,z)}{\vert z\vert ^k}.
\end{equation}
In fact, as well as in \eqref{H-id} we have $  \lim_{z\rightarrow -\infty}\dfrac{zA^{\prime}_{g}(z)}{A_{g}(z)} 
   =\lim_{ z\rightarrow -\infty}[1-zg(z)]
   \ge 1.$
Thus, 
	\begin{equation}\nonumber
    \begin{aligned}
       \lim_{z\rightarrow -\infty}\dfrac{e^{kG(z)}f(x,z)}{\left(\int_{z}^{0}e^{G(s)}ds\right)^{k}}
       &= \lim_{z\rightarrow -\infty} \Big[\frac{z A^{\prime}_g(z)}{A_g(z)}\Big]^{k}\frac{f(x,z)}{\vert z\vert ^{k}}
       \ge  \lim_{z\rightarrow -\infty}\frac{f(x,z)}{\vert z\vert ^{k}}.
    \end{aligned}
\end{equation}
Consequently, if $f(x,z)$ is superlinear then the pair $f(x,z),g(z)$ is superlinear in the sense of \eqref{f-superlinear}.
\end{remark}
\begin{lemma}\label{HAR} Assume that $f, g$ satisfy  the assumptions $(H_g)$, $(H_f)$, $(H_{SC})$ and $(H_{AR})$. Then, $h$ in \eqref{h-transformed} satisfies  \eqref{psi-subcritical} and \eqref{psi-AR}.
\end{lemma}
\begin{proof} If $t=A^{-1}_{g}(z)$ we have $t\rightarrow -\infty $ if $z\to -\infty$. Thus, for any $r>0$
\begin{equation}\nonumber
\begin{aligned}
\lim_{z\rightarrow -\infty}\dfrac{h(x,z)}{\vert z\vert^{r}}& =\lim_{z\rightarrow -\infty}\dfrac{e^{kG(A^{-1}_{g}(z))}f(x,A^{-1}_{g}(z))}{\vert z\vert^{r}}\\
&= \lim_{t\rightarrow -\infty}\dfrac{e^{kG(t)}f(x,t)}{\vert A_{g}(t)\vert^{r}}\\
&= \lim_{t\rightarrow -\infty}\dfrac{e^{kG(t)}f(x,t)}{\big(\int_{t}^{0}e^{G(s)}ds\big)^{r}}.
\end{aligned}
\end{equation}
This yields \eqref{psi-subcritical} provided that $f,g$ satisfies $(H_{SC})$. In addition,  assume that $f$ and $g$ satisfy $(H_{AR})$ for some $\theta>0$ and $M>0$. Thus, using the change of variable $t=A^{-1}_{g}(s)$ we get
\begin{equation}\nonumber
\begin{aligned}
    \int_{z}^{0}h(x,s)ds &=\int_{z}^{0}e^{kG(A^{-1}_{g}(s))}f(x,A^{-1}_{g}(s))ds\\
    &=\int_{A^{-1}_{g}(z)}^{0}e^{(k+1)G(t)}f(x,t)dt\\
    &\le  \frac{1-\theta}{k+1}e^{kG(A^{-1}_{g}(z))}f(x,A^{-1}_{g}(z))\int_{A^{-1}_{g}(z)}^{0}e^{G(s)}ds\\
     &=\frac{1-\theta}{k+1}h(x,z)\vert z\vert
\end{aligned}
\end{equation}
for any $z<A_g(-M)$. This proves \eqref{psi-AR}.
\end{proof} 

Analogous to the proof of Theorem~\ref{thm2-sup},  to prove Theorem~\ref{thm1-sub} we combine Proposition~\ref{proposição 2.1}  with the existence result in \cite[Theorem~1.3]{sobre k-hess3}. Indeed, by \cite[Theorem~1.3]{sobre k-hess3} the equation \eqref{HessianPsi} admits solution $u\in C^{3,\alpha}(\Omega)\cap C^{0}(\overline{\Omega})$ for some $\alpha\in (0,1)$ provided that  $\Omega$ is of class $C^{3,1}$, $\psi\in C^{1,1}(\overline{\Omega}\times\mathbb{ R^{-}})\cap C^{0}(\overline{\Omega\times\mathbb{ R^{-}}})$, with $\psi(x,z)>0$ for $z<0$
and satisfies 
\begin{equation}\label{psis02}
\lim_{ z\rightarrow 0^-}\dfrac{\psi(x,z)}{\vert z\vert ^k}>\lambda_{1}\;\; \mbox{uniformly on $\overline{\Omega}$},
\end{equation}
and
\begin{equation}\label{psi-sublinear}
\lim_{ z\rightarrow -\infty}\dfrac{\psi(x,z)}{\vert z\vert^k}<\lambda_{1}\;\; \mbox{uniformly on $\overline{\Omega}$}.
\end{equation}
We are assuming $h\in C^{1,1}(\overline{\Omega}\times\mathbb{R^{-}})\cap C^{0}(\overline{\Omega\times\mathbb{ R^{-}}})$ with $h(x,z)>0$ for $z<0$. Hence,  from Proposition~\ref{proposição 2.1} and \cite[Theorem~1.3]{sobre k-hess3} it is sufficient to show that $h$ given by \eqref{h-transformed} satisfies \eqref{psis02} and \eqref{psi-sublinear}. But, it is a easy consequence of \eqref{f-sublinear}, \eqref{s011} and Lemma~\ref{sub-super}. 

\section{Non-existence results: Proof of Theorem~\ref{non-ex}}
 In this section, we establish non-existence results for the $k$-Hessian equation in \eqref{equação 1.1}. First, let us consider the  equation
 \begin{equation}
	\begin{cases}\label{Gen-E}
		S_k[u]=\psi(x,u) & \mbox{in}\;\; \Omega\\
		\ u<0 & \mbox{in}\;\;  \Omega \\
		\ u=0 & \mbox{on}\;\; \partial \Omega,
	\end{cases}
\end{equation}
where $\psi:\overline{\Omega}\times (-\infty, 0]\to \mathbb{R}$ is a continuous  function. Based on the Pucci-Serrin general identity \cite{Pucci-Serrin}, the following result is established by Tso in \cite[Proposition~1]{tso}.
 
 \begin{proposition}\label{PropTSO}
Let $\Omega$ be a bounded $C^2$-domain which is star-shaped with respect to the origin.  Suppose that $\psi$ belongs to $C^{0}(\overline{\Omega\times\mathbb{R}^{-}})\cap C^1(\Omega \times \mathbb{R}^{-})$, positive in $\Omega \times\mathbb{R}^{-}$ and $\psi\equiv 0$ on $\Omega \times \{0\}$. Then, there are no  solutions to \eqref{Gen-E} which belong to $C^1(\overline{\Omega})\cap C^4(\Omega)$ if
 	$$
 	n\Psi(x,z)-\dfrac{n-2k}{k+1}z\psi(x,z)+x_i\Psi_{x_i}(x,z)>0, \;\; \mbox{on}\;\; \Omega \times (-\infty,0)
 	$$
 where $\Psi(x,z)=\int_{0}^{z}\psi(x,t)dt$. In addition, if $\langle x,\nu \rangle>0$ on $\partial \Omega$,  the same conclusion holds under
 	$$
 	n\Psi(x,z)-\dfrac{n-2k}{k+1}z\psi(x,z)+x_i\Psi_{x_i}(x,z)\geq0.
 	$$
 \end{proposition}
Combining Proposition~\ref{proposição 2.1} and Proposition~\ref{PropTSO}, we get the following non-existence result which  contains Theorem~\ref{non-ex}.
\begin{lemma}\label{lema-nonex}
Assume $\Omega$ under the conditions of Proposition~\ref{PropTSO}.  Suppose that $g:(-\infty, 0]\to [0, \infty)$ is a continuous function and $f\in C^{0}(\overline{\Omega\times\mathbb{R}^{-}})\cap C^1(\Omega \times \mathbb{R}^{-})$, positive in $\Omega \times\mathbb{R}^{-}$ and $f\equiv 0$ on $\Omega \times \{0\}$. Let 
\begin{equation}\nonumber
    h(x,s)=e^{kG(A^{-1}_{g}(s))}f(x,A^{-1}_{g}(s)),
\end{equation}
where $A_g$ is  given by \eqref{AG-change}. Then, \eqref{equação 1.1} has no negative solution in $C^1(\overline{\Omega})\cap C^4(\Omega)$ when 
		\begin{equation}
		nH(x,z)-\dfrac{n-2k}{k+1}zh(x,z)+x_iH_{x_i}(x,z)>0, \;\; \mbox{on}\;\; \Omega \times (-\infty,0)
		\end{equation}
	 with $H(x,z)=\int_{0}^{z}h(x,s)ds$. The same conclusion holds under
		\small
		\begin{equation}
		nH(x,z)-\dfrac{n-2k}{k+1}zh(x,z)+x_iH_{x_i}(x,z)\ge 0 
		\end{equation}
		if $\langle x,\nu\rangle>0$ on $\partial \Omega$.
\end{lemma}
\begin{proof}
The assumptions assumed on $g$ and $f$ ensure that $h$ satisfies all hypotheses of the Proposition~\ref{PropTSO}. Hence, we obtain the non-existence for the equation \eqref{equação 1.2}, and consequently from Proposition~\ref{proposição 2.1} we get the non-existence for the $k$-Hessian equation in \eqref{equação 1.1}.
\end{proof}

\paragraph{Proof of Theorem~\ref{non-ex}.} For $f(x,z)=(e^{-z}-1)^pe^{kz}$ and $g\equiv1$, we have $h(x,s)=(-s)^{p}$.
Thus, for $p+1>k^{*}=n(k+1)/(n-2k)$ we obtain
\begin{equation}\nonumber
\begin{aligned}
   	nH(x,z)-\dfrac{n-2k}{k+1}zh(x,z)+x_iH_{x_i}(x,z)
   	&=\Big[-\frac{n}{p+1}+\dfrac{n-2k}{k+1}\Big](-z)^{p+1}>0.
\end{aligned}
\end{equation}
If $p+1=k^{*}=n(k+1)/(n-2k)$, we get $	nH(x,z)-\dfrac{n-2k}{k+1}zh(x,z)+x_iH_{x_i}(x,z)=0$ and the non-existence still holds provided that the boundary assumption $\langle x,\nu\rangle>0$ on $\partial \Omega$ is assumed. Thus, Theorem~\ref{non-ex} follows from Lemma~\ref{lema-nonex}.

\section{Existence for the Trudinger-Moser case}
This section is devoted to prove the existence results for subcritical growth Theorem~\ref{thm2} and critical growth Theorem~\ref{thm3}. In order to get our aim we combine Proposition~\ref{proposição 2.1} with the existence results in \cite[Theorem~1.1 and Theorem~1.2]{crit e subcritico}. In \cite{crit e subcritico} the authors investigate existence of  radially symmetric $k$-admissible solution for $k$-Hessian equation 
\begin{equation}\label{E-}
	\begin{cases}
		S_k[u]=\varphi(x,-u) & \mbox{in}\;\; B\\
		\ u<0 & \mbox{in}\;\;  B \\
		\ u=0 & \mbox{on}\;\; \partial B,
	\end{cases}
\end{equation}
when $k=n/2$, $B\subset \mathbb{R}^n$ is the unit ball centered at origin and the function $\varphi:\overline{B}\times\mathbb{R}\to\mathbb{R}$ has subcritical exponential growth,i.e,   
\begin{equation}\label{varphi-e-sub}
    \lim_{s\to\infty}\vert \varphi(x,s)\vert e^{-\alpha\vert s\vert^{\frac{n+2}{n}}}=0,\;\;\mbox{uniformly for}\;\;x\in \overline{B},\;\;\mbox{for all}\;\; \alpha>0
\end{equation}
or critical exponential growth, i.e, there is $\alpha_0>0$ such that
\begin{equation}\label{e-cri}
\lim_{s\rightarrow\infty}\vert \varphi(x,s)\vert e^{-\alpha \vert s\vert^{\frac{n+2}{n}}}=\left\{
\begin{aligned}
           &0,   \;\;&\mbox{for all}&\;\;  \alpha>\alpha_0\\
		&\infty, \;\;&\mbox{for all}&\;\;  \alpha<\alpha_0  \\
\end{aligned}
\right.,\;\;\mbox{uniformly for}\;\;x\in \overline{B}.
\end{equation}
The following hypotheses on $\varphi$ were considered in \cite{crit e subcritico}.
\begin{enumerate}
    \item [$(\varphi_0)$] $\varphi$  is continuous and $\varphi\ge 0$ on $\overline{B}\times \mathbb{R}$, $\varphi(\cdot, s)$ is a radially symmetric function and $\varphi(x,s)=0$ for $(x,s)\in\overline{B}\times(-\infty, 0]$.
    \item [$(\varphi_1)$] There exist $\vartheta>k+1$, $0<r_1<r_2<1$ and $s_0>0$ such that for $s>s_0$
\begin{equation}\nonumber
\begin{aligned}
        \vartheta\Phi(x,s)\le s\varphi(x,s)\;\mbox{if}\;x\in\overline{B}\;\mbox{and}\;\Phi(x,s)>0 \;\mbox{if}\; x\in B_{r_2}\setminus B_{r_1},
        \end{aligned}
    \end{equation}
    where $\Phi(x,s)=\int_{0}^{s}\varphi(x,\tau)d\tau$.
     \item [$(\varphi_2)$] There exist $L>0$ and $M>0$ such that $0<\Phi(x,s)\le M \varphi(x,s)$, for $s>L$ and $x\in\overline{B}$.
    \end{enumerate}
With this notation,  the existence of radially symmetric $k$-admissible solution for \eqref{E-} is ensured in \cite{crit e subcritico} for the following conditions: 
\paragraph{\textbf{Case~1: Subcritical.}} If $\varphi$ has subcritical exponential growth \eqref{varphi-e-sub} and satisfies $(\varphi_0)$, $(\varphi_1)$ and
\begin{equation}\label{varphi-oriLam}
    \limsup_{s\to 0^{+}}\frac{(k+1)\Phi(x,s)}{\vert s\vert ^{k+1}}<\Lambda_1, \;\;\mbox{uniformly for}\;\;x\in \overline{B}
\end{equation}
where $\Phi(x,s)=\int_{0}^{s}\varphi(x,\tau)d\tau$ and $\Lambda_1$ is given by \eqref{Lambda1}.
\paragraph{\textbf{Case~2: Critical.}} If $\varphi$ has critical exponential growth \eqref{e-cri} and satisfies $(\varphi_0)$, $(\varphi_1)$, $(\varphi_2)$, \eqref{varphi-oriLam} and the estimate
\begin{equation}\label{min-max}
    \lim_{s \rightarrow \infty}\vert s\vert \varphi(x,s)e^{-\alpha_0 \vert s\vert ^{\frac{n+2}{n}}}=b_0>\dfrac{1}{e^{1+\frac{1}{2}+\cdots+\frac{1}{k}}}\left(\dfrac{\alpha_n}{\alpha_0}\right)^{\frac{n}{2}}\dfrac{n}{\tau},\;\mbox{uniformly for}\;x \in \overline{B}
\end{equation}
where $\alpha_n=n[\frac{\omega_{n-1}}{k}\binom{n-1}{k-1}]^{\frac{2}{n}}=nc_n^{\frac{2}{n}}$ is the $k$-Hessian critical Moser's constant (cf. \cite{Tian}) and $\tau=\omega_{n-1}$ is the surface area of the unit sphere in $\mathbb{R}^n$.

Since the hypotheses on $\varphi$ in the equation \eqref{E-} were considered on $\overline{B}\times[0,\infty)$ and our problem here \eqref{equação 1.1} is posed on $\overline{B}\times(-\infty, 0]$, we need to translate that assumptions  for this context. The next two lemmas are just to fulfill this role.

\begin{lemma}\label{lf1} Let $h:\overline{B}\times \mathbb{R}\to \mathbb{R}$ be a continuous function  and set $H(x,s)=\int_{s}^{0}h(x,\tau)d\tau$. Then, the function $\varphi:\overline{B}\times \mathbb{R}\to \mathbb{R}$ given by $\varphi(x,s)=h(x,-s)$ satisfies
    \begin{equation}\nonumber
    \left\{\begin{aligned}
        &\lim_{s\to+\infty}\vert \varphi(x,s)\vert e^{-\alpha\vert s\vert^{\frac{n+2}{n}}}=\lim_{s\to-\infty}\vert h(x,s)\vert e^{-\alpha\vert s\vert ^{\frac{n+2}{n}}}\\
        &\lim_{s\to+\infty}s\varphi(x,s)e^{-\alpha\vert s\vert^{\frac{n+2}{n}}}=\lim_{s\to-\infty}\vert s\vert h(x,s)e^{-\alpha\vert s\vert^{\frac{n+2}{n}}}\\
        &\limsup_{s\to 0^{+}}\frac{(k+1)\Phi(x,s)}{\vert s\vert^{k+1}}=\limsup_{s\to 0^{-}}\frac{(k+1)H(x,s)}{\vert s\vert^{k+1}},
        \end{aligned}\right.
    \end{equation}
where $\Phi(x,s)=\int_{0}^{s}\varphi(x,\tau)d\tau$.
\end{lemma}
\begin{proof}
Note that for each $s\in\mathbb{R}$ one has
    \begin{equation}\label{HPhi}
        H(x,s)=-\int_{-s}^{0}h(x,-t)dt=\int_{0}^{-s}h(x,-t)dt=\int_{0}^{-s}\varphi(x,t)dt=\Phi(x,-s).
    \end{equation}
Thus, the result follows from  the change of variables $t=-s$.
\end{proof}

\begin{lemma}\label{lf2} Let $h:\overline{B}\times \mathbb{R}\to \mathbb{R}$ be a continuous function and consider the following conditions:
\begin{enumerate}
\item [$(h_0)$] $h\ge 0$ on $\overline{B}\times \mathbb{R}$, $h(\cdot, s)$ is a radially symmetric function and $h(x,s)=0$ for $(x,s)\in\overline{B}\times[0,\infty)$
\item [$(h_1)$] There exist $\vartheta>k+1$, $0<r_1<r_2<1$ and $s_0>0$ such that for $s<-s_0$
\begin{equation}\nonumber
\begin{aligned}
         \vartheta H(x,s)\le \vert s\vert h(x,s)\;\mbox{if} \; x\in\overline{B}\; \mbox{and} \; H(x,s)>0 \;\mbox{if} \; x\in B_{r_2}\setminus B_{r_1},
        \end{aligned}
    \end{equation}
    with $H(x,s)=\int_{s}^{0}h(x,\tau)d\tau.$
    \item [$(h_2)$] There exist $L>0$ and $M>0$ such that $0<H(x,s)\le M h(x,s)$, for $s<-L$ and $x\in\overline{B}$. 
    \end{enumerate}
    Then, $\varphi:\overline{B}\times \mathbb{R}\to \mathbb{R}$ defined by $\varphi(x,s)=h(x,-s)$ is such that each condition $(h_i)$ implies the corresponding hypothesis $(\varphi_i)$, for $i=0,1,2$.
\end{lemma}
\begin{proof} Of course, $(h_0)$ implies $(\varphi_0)$. In addition, from \eqref{HPhi} we also have  $H(x,-s)=\Phi(x,s).$ Thus, for $s>s_0$ we have $-s<-s_0$ and from $(h_1)$, we get 
$\vartheta H(x,-s)\le \vert s\vert h(x,-s)$ if $x\in\overline{B}$ and  $H(x,-s)>0$ if $ x\in B_{r_2}\setminus B_{r_1}$ or equivalently $\vartheta\Phi(x,s)\le s\varphi(x,s)$ if $x\in\overline{B}$ and $\Phi(x,s)>0$ if $ x\in B_{r_2}\setminus B_{r_1}$. This proves  that $(h_1)$ implies $(\varphi_1)$. Analogously, for $s>L$ we have $-s<-L$ and $(h_2)$ yields $(\varphi_2)$.
\end{proof}
\begin{lemma}\label{H-FG} Assume $(H_g)$ and $(\mathcal{H}_f)$. Let $h$ be  given by \eqref{h-transformed} and $H(x,s)=\int_{s}^{0}h(x,\tau)d\tau$. Then
\begin{enumerate}
    \item [$(a)$] \begin{equation}\nonumber
\begin{aligned}
   \lim_{s\to-\infty} \vert h(x,s)\vert e^{-\alpha\vert s\vert ^{\frac{n+2}{n}}}
  &= \lim_{z\to-\infty}e^{kG(z)-\alpha\vert A_{g}(z)\vert ^{\frac{n+2}{n}}}f(x,z).
   \end{aligned}
\end{equation}
\item [$(b)$]  \begin{equation}\nonumber
\begin{aligned}
   \limsup_{s\to 0^{-}}\frac{(k+1)H(x,s)}{\vert s\vert ^{k+1}}
   \le \limsup_{t\to 0^{-}}\frac{e^{kG(t)}f(x,t)}{(\int_{t}^{0}e^{G(t)}d \tau)^{k}}.
   \end{aligned}
\end{equation}
\item [$(c)$] If $f,g$ satisfies $(\mathcal{H}_{AR})$ then  for any $s<A_{g}(-z_0)$ we have
\begin{equation}\nonumber
    \begin{aligned}
    \vartheta H(x,s)\le \vert s\vert h(x,s) \;\;\mbox{if}\; x\in \overline{B}\;\;\mbox{and}\;\; H(x,s)>0\;\;\mbox{if}\;\;x\in B_{r_2}\setminus B_{r_{1}}.
    \end{aligned}
\end{equation}
\item [$(d)$] If $f,g$ satisfies $(\mathcal{H}_{AR_1})$ then  $0<H(x,s)\le M h(x,s)$, for $s<A_{g}(-L)$ and $x\in\overline{B}$.
\end{enumerate}
\end{lemma}
\begin{proof}
$(a).$ By setting $z=A^{-1}_{g}(s)$ we have $z\to-\infty$ if $s\to -\infty$. Hence 
\begin{equation}\nonumber
\begin{aligned}
   \lim_{s\to-\infty} \vert h(x,s)\vert e^{-\alpha\vert s\vert ^{\frac{n+2}{n}}}& = \lim_{s\to-\infty}e^{kG(A^{-1}_{g}(s))-\alpha\vert s\vert ^{\frac{n+2}{n}}}f(x,A^{-1}_{g}(s))\\
   &= \lim_{z\to-\infty}e^{kG(z)-\alpha\vert A_{g}(z)\vert ^{\frac{n+2}{n}}}f(x,z).
   \end{aligned}
\end{equation}
$(b).$ Analogously, since we have $\limsup (f/g)\le \limsup(f^{\prime}/g^{\prime})$ it follows that
\begin{equation}\nonumber
\begin{aligned}
   \limsup_{s\to 0^{-}}\frac{(k+1)H(x,s)}{\vert s\vert ^{k+1}}&=\limsup_{s\to 0^{-}}\frac{(k+1)\int_{A^{-1}_{g}(s)}^{0}e^{(k+1)G(z)}f(x,z)dz}{\vert s\vert ^{k+1}}\\
   &=\limsup_{t\to 0^{-}}\frac{(k+1)\int_{t}^{0}e^{(k+1)G(z)}f(x,z)dz}{(\int_{t}^{0}e^{G(\tau)}d \tau)^{k+1}}\\
   &\le \limsup_{t\to 0^{-}}\frac{e^{kG(t)}f(x,t)}{(\int_{t}^{0}e^{G(t)}d \tau)^{k}}.
   \end{aligned}
\end{equation}
Note that 
\begin{equation}\label{e-H}
  H(x,s)= \int_{s}^{0}e^{kG(A^{-1}_{g}(t))}f(x,A^{-1}_{g}(t)) dt= \int_{A^{-1}_{g}(s)}^{0}e^{(k+1)G(t)}f(x,t)dt.
\end{equation}
$(c).$ For any $s<A_{g}(-z_0)$  and $x\in\overline{B}$ the assumption $(\mathcal{H}_{AR})$ and \eqref{e-H} yield
\begin{equation}\nonumber
    \begin{aligned}
    \vartheta H(x,s)&=\vartheta \int_{A^{-1}_{g}(s)}^{0}e^{(k+1)G(t)}f(x,t)ds  \\
    &\le e^{kG(A^{-1}_{g}(s))}f(x,A^{-1}_{g}(s))\int_{A^{-1}_{g}(s)}^{0}e^{G(t)}dt\\
    &=\vert s\vert h(x,s)
    \end{aligned}
\end{equation}
and $H(x,s)>0$, for $x\in B_{r_2}\setminus B_{r_1}$. 

\noindent $(d).$  For $s<A_{g}(-L)$ and $x\in\overline{B}$,  $(\mathcal{H}_{{AR}_{1}})$ and \eqref{e-H} yield
\begin{equation}\nonumber
    \begin{aligned}
   0< H(x,s)=\int_{A^{-1}_{g}(s)}^{0}e^{(k+1)G(t)}f(x,t)dt\le Me^{kG(A^{-1}_{g}(s))}f(x,A^{-1}_{g}(s))=Mh(x,s).
    \end{aligned}
\end{equation}
\end{proof}

\begin{proof}[Proof of Theorem~\ref{thm2}]
    We will combine Proposition~\ref{proposição 2.1} with the existence result  for subcritical growth,  see \textbf{Case~1} above. Thus, by using Lemma~\ref{lf1} and Lemma~\ref{lf2} we only need to check that the transformed function $h$ in \eqref{h-transformed} satisfies  $(h_0), (h_1)$, the condition \eqref{varphi-e-sub} with $s\to -\infty$ instead of $s\to +\infty$, and the assumption \eqref{varphi-oriLam} with $s\to 0^{-}$ and $H(x,s)=\int_{s}^{0}h(x,\tau)d\tau$. First, from $(H_g)$ and $(\mathcal{H}_f)$, the function $h(x,s)$ is continuous and nonnegative on $\overline{B}\times (-\infty, 0]$. Also, since $A^{-1}_g(0)=0$ and $f(x,0)=0$ we have $h(x, 0)=0$. Thus, by replacing $h$ by its continuous extension to be zero on $\overline{B}\times [0,\infty)$ we can assume $(h_0)$.  The condition $(h_1)$ follows from Lemma~\ref{H-FG}-$(c)$. The subcritical condition \eqref{varphi-e-sub} follows from \eqref{e-sub} and  Lemma~\ref{H-FG}-$(a)$. Finally,  Lemma~\ref{H-FG}-$(b)$ and \eqref{super00} yield \eqref{varphi-oriLam}.
\end{proof}
 
\begin{proof}[Proof of Theorem~\ref{thm3}]
    Here, we also combine Proposition~\ref{proposição 2.1} with the existence result  for critical growth,  see \textbf{Case~2} above. As in the previous argument, from Lemma~\ref{lf1} and Lemma~\ref{lf2} it is sufficient to verify that $h$ in \eqref{h-transformed} satisfies the hypotheses $(h_0), (h_1)$, $(h_2)$, the critical growth condition \eqref{e-cri} and the estimate \eqref{min-max} with $s\to -\infty$ instead of $s\to +\infty$,  and the origin assumption \eqref{varphi-oriLam} with $s\to 0^{-}$. The hypotheses $(h_0)$ and $(h_1)$ and the condition \eqref{varphi-oriLam}  were verified in the proof of Theorem~\ref{thm2} above. In addition, $(h_2)$ follows from Lemma~\ref{H-FG}-$(d)$. Finally, the estimate \eqref{min-max} and the critical condition \eqref{e-cri} follow from  
 \eqref{fe-cri}, \eqref{fmin-max} and  Lemma~\ref{H-FG}-$(a)$.
\end{proof}

\bigskip
\bigskip
\end{document}